\documentclass[12pt,leqno,a4paper]{amsart}
\usepackage[T1]{fontenc}
\usepackage[utf8]{inputenc}
\usepackage{lmodern}
\usepackage{amsmath,amssymb,amsthm}
\usepackage[margin=28mm]{geometry}
\usepackage{microtype}
\usepackage{needspace}
\usepackage[hidelinks]{hyperref}
\usepackage{url}
\allowdisplaybreaks[2]
\numberwithin{equation}{section}
\newtheorem{thm}{Theorem}[section]
\newtheorem{lem}[thm]{Lemma}
\newtheorem{prop}[thm]{Proposition}

\theoremstyle{definition}

\newtheorem{exa}{Example}
\theoremstyle{remark}
\newtheorem{rmk}{Remark}
\newcommand{\R}{\mathbb R}
\newcommand{\C}{\mathbb C}
\newcommand{\N}{\mathbb N}
\newcommand{\Nzero}{\mathbb N_0}
\newcommand{\Q}{\mathbb Q}
\newcommand{\Z}{\mathbb Z}
\newcommand{\D}{\mathcal D}
\newcommand{\dd}{\,\mathrm d}
\newcommand{\ip}[2]{\left(#1,#2\right)}

\newcommand{\doi}[1]{\href{https://doi.org/#1}{\nolinkurl{#1}}}
\newcommand{\arxiv}[1]{\href{https://arxiv.org/abs/#1}{arXiv:\nolinkurl{#1}}}

\title[Source factor and fractional order with unknown initial data]{Uniqueness in determining the temporal source factor and the fractional order with an unknown initial value}
\author[R. Ashurov]{Ravshan Ashurov}
\address{V.~I.~Romanovskiy Institute of Mathematics, Uzbekistan Academy of Sciences, Tashkent 100174, Uzbekistan; School of Engineering, Central Asian University, Tashkent 111221, Uzbekistan}
\email{ashurovr@gmail.com}
\author[M. Yamamoto]{Masahiro Yamamoto}
\address{Graduate School of Mathematical Sciences, The University of Tokyo, Komaba, Meguro, Tokyo 153-8914, Japan; Department of Mathematics, Faculty of Arts and Sciences, Zonguldak B\"ulent Ecevit University, Zonguldak 67100, Turkey}
\email{myama@ms.u-tokyo.ac.jp}
\date{}

\begin{document}
\begin{abstract}
We consider inverse problems for a time-fractional diffusion equation with a second-order elliptic operator and the homogeneous Dirichlet boundary condition in a bounded domain of $\mathbb R^N$. The initial value is unknown and is not assumed to be zero. The main result concerns the simultaneous determination of the fractional order and the time-dependent source factor from one weighted spatial observation. For irrational orders in $(0,1)$ and bounded source factors satisfying an exponential integrability condition, we prove that a nonzero observation given for all positive times uniquely determines both unknowns. We also describe the information about the initial value determined by this observation. A counterexample shows that uniqueness of the source factor may fail at the classical order $\alpha=1$, even when the observation is nonzero. In the second part, we establish two uniqueness results for the fractional order when the elliptic operators, initial values, and source terms in the compared equations may be different. The first result is proved by the Laplace transform, whereas the second uses short-time observations and the asymptotic behavior of the Mittag--Leffler functions.
\end{abstract}
\maketitle

\section{Introduction}\label{sec:introduction}

Let $\Omega\subset\R^N$, $N\geq1$, be a bounded connected domain with a smooth boundary $\partial\Omega$. Consider the second-order elliptic operator
\begin{equation}\label{eq:elliptic-A}
 Av(x)=-\sum_{i,j=1}^N\partial_i\bigl(a_{ij}(x)\partial_jv(x)\bigr)-c(x)v(x),
 \qquad x\in\Omega,
\end{equation}
where $a_{ij}=a_{ji}\in C^1(\overline\Omega)$, $c\in C^1(\overline\Omega)$, $c\leq0$, and
\begin{equation}\label{eq:ellipticity-A}
 \sum_{i,j=1}^N a_{ij}(x)\zeta_i\zeta_j\geq\kappa_A|\zeta|^2,
 \qquad x\in\overline\Omega,\quad \zeta\in\R^N,
 \qquad \kappa_A>0.
\end{equation}
We use the same notation $A$ for the differential expression (\ref{eq:elliptic-A}) with a homogeneous Dirichlet condition, and with
\[
 \D(A)=H^2(\Omega)\cap H_0^1(\Omega).
\]
All coefficients and spatial data in this paper are real-valued. We denote the scalar product and norm in $L^2(\Omega)$ by $(\cdot,\cdot)$ and $\|\cdot\|$, respectively. Thus
\[
 (v,w)=\int_\Omega v(x)w(x)\dd x.
\]
Under the above assumptions, $A$ is a positive self-adjoint operator with a compact inverse; see, for example, \cite{4,6,7}.

Let $0<\alpha<1$. We consider the initial-boundary value problem
\begin{equation}\label{1}
 \begin{cases}
 \partial_t^\alpha\bigl(u(x,t)-a(x)\bigr)+Au(x,t)=\mu(t)f(x),
       &x\in\Omega,\quad t>0,\\
 u(x,t)=0,&x\in\partial\Omega,\quad t>0,
 \end{cases}
\end{equation}
where $a\in H_0^1(\Omega)$ and $f\in L^2(\Omega)$. The fractional derivative is understood in the sense of Yamamoto \cite{7}. In particular, the initial condition is expressed by
\begin{equation}\label{eq:initial-space}
 u-a\in H_\alpha(0,T;L^2(\Omega))\qquad\hbox{for every }T>0.
\end{equation}
The equation is satisfied in $L^2(0,T;L^2(\Omega))$, and
\[
 u\in L^2(0,T;H^2(\Omega)\cap H_0^1(\Omega)).
\]
We recall the definition of $H_\alpha$ and the fractional derivative in Section~\ref{sec:preliminaries}. Throughout the paper, the temporal source factors are essentially bounded on the time intervals under consideration. For such factors, the solution is continuous from $[0,T]$ to $L^2(\Omega)$ and satisfies $u(\cdot,0)=a$; see Theorem \ref{lem2} below.

The spatial factor $f$ is prescribed, while the fractional order $\alpha$, the temporal factor $\mu$, and the initial value $a$ are unknown. Our observation is the weighted spatial average
\begin{equation}\label{eq:observation}
 y(t)=\int_\Omega u(x,t)\psi(x)\dd x,
\end{equation}
where $\psi\in L^2(\Omega)$ is a fixed weight function. We study two inverse problems associated with these data.

Denote the solution of \eqref{1}--\eqref{eq:initial-space} by $u_{\alpha,\mu,a}$. For another order $\beta\in(0,1)$, initial value $b\in H_0^1(\Omega)$, and temporal factor $\nu$, let $u_{\beta,\nu,b}$ denote the solution with the same operator $A$ and the same spatial source factor $f$.

\medskip
\noindent\textbf{Inverse Problem 1.}
Does the identity
\begin{equation}\label{eq:intro-observation}
 \int_\Omega u_{\alpha,\mu,a}(x,t)\psi(x)\dd x
 =\int_\Omega u_{\beta,\nu,b}(x,t)\psi(x)\dd x,
 \qquad t>0,
\end{equation}
imply that $\alpha=\beta$ and $\mu=\nu$?
\medskip

The principal difficulty is that the initial values $a$ and $b$ are unknown and need not coincide. Theorem~\ref{thm1} gives an affirmative answer for irrational orders, under the conditions
\[
 (f,\psi)\ne0,\qquad
 e^{p_0t}\mu,\ e^{p_0t}\nu\in L^1(0,\infty)
 \quad\hbox{for some }p_0>0,
\]
provided that the common observation is not identically zero. No sign condition on $a$, $b$, $f$, or $\psi$ is imposed in this result. Moreover, the observation determines the scalar quantities $(P_na,\psi)$, where $P_n$ is the orthogonal projection onto an eigenspace of $A$. In general, this does not determine the entire initial value. The precise conclusion and a sufficient condition for $a=b$ are given below.

For the second inverse problem, let
\begin{equation}\label{eq:elliptic-B}
 Bv(x)=-\sum_{i,j=1}^N\partial_i\bigl(b_{ij}(x)\partial_jv(x)\bigr)-d(x)v(x),
 \qquad \D(B)=H^2(\Omega)\cap H_0^1(\Omega),
\end{equation}
where $b_{ij}=b_{ji}\in C^1(\overline\Omega)$, $d\in C^1(\overline\Omega)$, $d\leq0$, and
\[
 \sum_{i,j=1}^N b_{ij}(x)\zeta_i\zeta_j\geq\kappa_B|\zeta|^2,
 \qquad x\in\overline\Omega,\quad\zeta\in\R^N,
 \qquad\kappa_B>0.
\]
Consider also the problem
\begin{equation}\label{3}
 \begin{cases}
 \partial_t^\beta\bigl(v(x,t)-b(x)\bigr)+Bv(x,t)=\nu(t)g(x),
       &x\in\Omega,\quad t>0,\\
 v(x,t)=0,&x\in\partial\Omega,\quad t>0,
 \end{cases}
\end{equation}
with $v-b\in H_\beta(0,T;L^2(\Omega))$ for every $T>0$, where $b\in H_0^1(\Omega)$ and $g\in L^2(\Omega)$. To indicate all the data, write the solutions as $u_{A,a,\mu,f,\alpha}$ and $u_{B,b,\nu,g,\beta}$.

\medskip
\noindent\textbf{Inverse Problem 2.}
Does the identity
\[
 \int_\Omega u_{A,a,\mu,f,\alpha}(x,t)\psi(x)\dd x
 =\int_\Omega u_{B,b,\nu,g,\beta}(x,t)\psi(x)\dd x,
 \qquad t>0,
\]
imply that $\alpha=\beta$, even when the operators, initial values, and source terms are not known to coincide?
\medskip

We obtain two results for this problem. The first uses the Laplace transform and nonvanishing weighted averages of $A^{-1}a$ and $B^{-1}b$. The second uses the first nonzero fractional term in the short-time expansion of the observation. In the latter result, observations on an arbitrary interval $(0,T)$ are sufficient, and no assumption on the source factors for large times is needed. The assumptions of these two results are different. Both results allow arbitrary orders in $(0,1)$, including rational orders.

Time-fractional diffusion equations are used to describe diffusion processes with memory; see, for example, \cite{Baleanu,Mainardi,MK,SKB}. The fractional order characterizes the temporal memory, whereas $\mu(t)$ describes the variation of the source intensity. The determination of these quantities from additional information about the solution leads to inverse problems. General background on inverse problems can be found in \cite{Kabanikhin}; surveys and related results for fractional equations are given in \cite{LiYamamoto1,LiYamamoto2,LiYamamoto3,Kirane,Ruzhansky}.

Various observation conditions have been considered for determining fractional orders. Weighted observations with a first-eigenfunction weight were used in \cite{AU2020}, and harmonic weights in a whole-space problem were considered in \cite{PskhuUzmat}. Order determination when other model parameters are unknown is studied in \cite{JK,Ya21,Yamamoto-arxiv}. Inverse problems for sources and initial values using observations after the source has ceased to act are considered in \cite{Loreti,Ya22}.

The work most closely related to the present paper is \cite{AshurovYamamoto}, where simultaneous determination of the order and the temporal source factor is studied for zero initial data, using the behavior of the source near $t=0$. That work also gives a counterexample for nonzero initial data. In the present paper, the initial value is unknown, but the assumptions on the observations and source factors are different: the main result uses data for all positive times and exponential integrability of the source factors. Under these assumptions, the initial contribution and the source contribution can be distinguished. Thus the unknown initial value does not prevent simultaneous uniqueness. The zero-observation case is described in Remark~\ref{rem:zero-observation}.

The remainder of the paper is organized as follows. Section~\ref{sec:preliminaries} contains the required definitions, auxiliary results, the forward theorem of Yamamoto, and the main uniqueness theorems. Section~\ref{sec:simultaneous} is devoted to Inverse Problem~1 and includes a counterexample at the classical order $\alpha=1$. Section~\ref{sec:order} contains the two proofs for Inverse Problem~2. Some concluding remarks are given in Section~\ref{sec:conclusion}.

\section{Preliminaries and main results}\label{sec:preliminaries}

In this section, we first recall the fractional derivative used in \eqref{1} and several properties of the Mittag--Leffler functions. We then state the forward theorem and the main results.

For $0<\alpha<1$ and $w\in L^2(0,T;L^2(\Omega))$, the Riemann--Liouville fractional integral is defined by
\[
 J^\alpha w(t)=\frac1{\Gamma(\alpha)}
       \int_0^t(t-s)^{\alpha-1}w(s)\dd s,
\]
where the integral is understood in the Bochner sense. Following Yamamoto \cite[Section~2]{7}, we set
\begin{equation}\label{eq:inverse-integral}
 \begin{gathered}
 H_\alpha(0,T;L^2(\Omega))=J^\alpha L^2(0,T;L^2(\Omega)),\\
 \partial_t^\alpha=(J^\alpha)^{-1},\qquad
 \D(\partial_t^\alpha)=H_\alpha(0,T;L^2(\Omega)).
 \end{gathered}
\end{equation}
The norm in this space is
\[
 \|w\|_{H_\alpha(0,T;L^2(\Omega))}
   =\|\partial_t^\alpha w\|_{L^2(0,T;L^2(\Omega))}.
\]
This is the restriction of Yamamoto's extended fractional derivative to functions whose derivative belongs to $L^2(0,T;L^2(\Omega))$. No pointwise time derivative of $w$ is required.

The space $H_\alpha$ has the following characterization, with equivalent natural norms; see \cite{4,7}:
\begin{equation}\label{Halpha}
 H_\alpha(0,T;L^2(\Omega))=
 \begin{cases}
 H^\alpha(0,T;L^2(\Omega)),&0<\alpha<\dfrac12,\\[4pt]
 \displaystyle\left\{w\in H^{1/2}(0,T;L^2(\Omega)):
             \int_0^T\frac{\|w(t)\|^2}{t}\dd t<\infty\right\},
          &\alpha=\dfrac12,\\[8pt]
 \left\{w\in H^\alpha(0,T;L^2(\Omega)):w(0)=0\right\},
          &\dfrac12<\alpha<1.
 \end{cases}
\end{equation}
Here $H^\alpha$ is the usual Sobolev--Slobodeckij space; see \cite{1}. For sufficiently smooth $u$ with $u(\cdot,0)=a$, the expression $\partial_t^\alpha(u-a)$ coincides with the classical Caputo derivative of $u$.

For $\gamma>0$ and $\eta\in\C$, the two-parameter Mittag--Leffler function is defined by
\[
 E_{\gamma,\eta}(z)=\sum_{k=0}^{\infty}
       \frac{z^k}{\Gamma(\gamma k+\eta)},\qquad z\in\C.
\]
We write $E_\gamma=E_{\gamma,1}$. Let us recall the properties needed below.
\begin{lem}\label{LaplasMLfunc}
Let $0<\gamma<1$, $\lambda>0$. Then $E_\gamma(-\lambda t)$ is complete monotone and
\begin{equation}\label{E0_1}
    0<E_\gamma(-\lambda t)\leq1,\quad\text{for}\quad t\geq0.
\end{equation}
For $\eta\in\mathbb{R}$, there is a constant $C_{\gamma,\eta}$ such that
\begin{equation}\label{eq:ML-bound}
 |E_{\gamma,\eta}(-t)|\leq\frac{C_{\gamma,\eta}}{1+t},\qquad t\geq0.
\end{equation}

Let
\[
 k_{\gamma,\lambda}(t)
   =t^{\gamma-1}E_{\gamma,\gamma}(-\lambda t^\gamma),\qquad t>0.
\]
Then
$k_{\gamma,\lambda}\geq0$, and
\begin{equation}\label{eq:kernel-mass}
 \int_0^t k_{\gamma,\lambda}(s)\dd s
     =\frac{1-E_\gamma(-\lambda t^\gamma)}{\lambda}
     \leq\frac1\lambda.
\end{equation}
Moreover
\begin{equation}\label{L1}
 k_{\gamma,\lambda}\in L^1(0,\infty),\qquad
 \int_0^\infty k_{\gamma,\lambda}(t)\dd t\leq\frac1\lambda.
\end{equation}

\end{lem}

\begin{proof}
The proof of the complete monotonicity and the estimates (\ref{E0_1}) and (\ref{eq:ML-bound}) can be found, for example, in \cite{5}. The formula (\ref{eq:kernel-mass}) can be found, for example, in \cite{5} and a proof follows directly from the equality
\[
\frac{d}{dt} E_\gamma(-\lambda t^\gamma)=-\lambda k_{\gamma,\lambda}(t),\qquad t>0,
\]
which can be verified, for example, by (1.83) on p. 22 of \cite{5}.
Letting $t\to\infty$ in (\ref{eq:kernel-mass}), and using the monotone convergence theorem, we
obtain (\ref{L1}).
\end{proof}

Let us denote the Laplace transform of a given bounded function $f \in C[0, \infty)$  by 
\[
\widehat{f}(s) =\int_0^\infty e^{-st} f(t) \, dt,\,\, \Re s > 0.
\]
If the functions $f$ and $g$ are of exponential growth type; that is, there exist constants $M_1, M_2 > 0$ and $a_1, a_2 \ge 0$ such that: $$
\vert{}f(t)\vert{} \le M_1 e^{a_1 t}, \quad \vert{}g(t)\vert{} \le M_2 e^{a_2 t} \quad \text{for all } t \ge 0,
$$ 
then the convolution 
\begin{equation*}\label{con}
(f * g)(t) = \int_{0}^{t} f(\tau) g(t - \tau) \, d\tau.
\end{equation*}
is also continuous and of exponential growth, and the formula 
\begin{equation}\label{conLap}
\widehat{(f * g)}(s) = \widehat{f}(s) \cdot \widehat{g}(s),
\end{equation}
holds for all $s$ with $\Re s > \max(a_1, a_2)$.

The standard term-by-term derivation of the following Laplace
transform formulas, as presented in \cite{5}, requires the restriction
$p>\lambda^{1/\gamma}$. In the present paper we need these formulas
for every $p>0$, since we will let $p\downarrow0$. We therefore give
a detailed proof of the following lemma.

\begin{lem}\label{lem:laplace-ML}
Let $0<\gamma<1$ and $\lambda>0$. Set
\[
w(t) = E_\gamma(-\lambda t^\gamma), \qquad k_{\gamma,\lambda}(t)
   =t^{\gamma-1}E_{\gamma,\gamma}(-\lambda t^\gamma),
 \qquad t>0.
\]
Then, for every $p\in\C$ with $\operatorname{Re}p>0$, both integrals below
converge absolutely and
\begin{equation}\label{eq:ML-transforms}
 \begin{aligned}
 \int_0^\infty e^{-pt}w(t)\dd t
    &=\frac{p^{\gamma-1}}{p^\gamma+\lambda},\\
 \int_0^\infty e^{-pt}k_{\gamma,\lambda}(t)\dd t
    &=\frac{1}{p^\gamma+\lambda}.
 \end{aligned}
\end{equation}
The powers of $p$ are taken on the principal branch.
\end{lem}

\begin{proof}

Let $\sigma=\operatorname{Re}p>0$. Then by Lemma \ref{LaplasMLfunc}
\[
 \begin{aligned}
 \int_0^\infty |e^{-pt}w(t)|\dd t
     &\leq\int_0^\infty e^{-\sigma t}\dd t=\frac1\sigma,\\
 \int_0^\infty |e^{-pt}k_{\gamma,\lambda}(t)|\dd t
     &\leq\int_0^\infty k_{\gamma,\lambda}(t)\dd t
       \leq\frac1\lambda.
 \end{aligned}
\]
Thus both Laplace integrals exist absolutely throughout the half-plane
$\operatorname{Re}p>0$.

\medskip
\noindent\emph{Next, we derive the first formula for
$p>\lambda^{1/\gamma}$ (see, e.g, \cite{5}).}
At this stage $p$ is real. The substitution $s=pt$ in Euler's gamma
integral yields
\[
 \int_0^\infty e^{-pt}t^{\gamma m}\dd t
       =\frac{\Gamma(\gamma m+1)}{p^{\gamma m+1}},
 \qquad m=0,1,\ldots.
\]
Consequently, when $p>\lambda^{1/\gamma}$,
\[
 \begin{aligned}
 \sum_{m=0}^{\infty}\int_0^\infty
 e^{-pt}\frac{\lambda^m t^{\gamma m}}{\Gamma(\gamma m+1)}\dd t
 &=\frac1p\sum_{m=0}^{\infty}
               \left(\frac{\lambda}{p^\gamma}\right)^m\\
 &=\frac{1}{p\bigl(1-\lambda p^{-\gamma}\bigr)}<\infty.
 \end{aligned}
\]
This is the absolute-integrability condition needed to interchange the
power series with the improper integral. Therefore,
\[
 \begin{aligned}
 F(p)&:=\int_0^\infty e^{-pt}E_\gamma(-\lambda t^\gamma)\dd t\\
 &=\sum_{m=0}^{\infty}
     \frac{(-\lambda)^m}{\Gamma(\gamma m+1)}
          \int_0^\infty e^{-pt}t^{\gamma m}\dd t\\
 &=\frac1p\sum_{m=0}^{\infty}
             \left(-\frac{\lambda}{p^\gamma}\right)^m
   =\frac1{p(1+\lambda p^{-\gamma})}
   =\frac{p^{\gamma-1}}{p^\gamma+\lambda}.
 \end{aligned}
\]
This is the usual series calculation behind the formula with the
restriction $p>\lambda^{1/\gamma}$; see \cite{5}.
For $0<p\leq\lambda^{1/\gamma}$, the geometric series just obtained
does not converge, so this particular calculation
cannot be used. It does not follow that the Laplace integral fails to
exist there: its absolute convergence has already been proved above.

\medskip
\noindent\emph{We now show that $F$ is holomorphic in
$\operatorname{Re}p>0$.}
Let $w(t) = E_\gamma(-\lambda t^\gamma)$ and
\[
 \mathcal H=\{p\in\C:\operatorname{Re}p>0\}.
\]
Fix an arbitrary $p_*\in\mathcal H$ and put
$\sigma_*=(\operatorname{Re}p_*)/2>0$. For $0<|h|<\sigma_*$,
\[
 \frac{e^{-(p_*+h)t}-e^{-p_*t}}{h}
      =-t\int_0^1 e^{-(p_*+\theta h)t}\dd\theta.
\]
Since $\operatorname{Re}(p_*+\theta h)\geq\sigma_*$ and
$|w(t)|\leq1$, we have
\[
 \left|\frac{e^{-(p_*+h)t}-e^{-p_*t}}{h}w(t)\right|
       \leq t e^{-\sigma_*t},\qquad
 \int_0^\infty t e^{-\sigma_*t}\dd t=\frac1{\sigma_*^2}.
\]
The majorant is independent of $h$ and integrable. Passing to the
limit under the integral by dominated convergence, we obtain the
complex derivative
\[
 F'(p_*)=-\int_0^\infty t e^{-p_*t}w(t)\dd t.
\]
Since $p_*$ was arbitrary, $F$ is holomorphic in $\mathcal H$.

\medskip
\noindent\emph{We also verify holomorphy of the right-hand side.}
To apply the identity theorem, it is not enough to know that $F$ is
holomorphic: the function
\[
 R(p)=\frac{p^{\gamma-1}}{p^\gamma+\lambda}
\]
must be a single-valued holomorphic function in the same half-plane.
For positive real $p$, its fractional powers have their usual meaning.
For complex $p$, we specify the branch of these powers as follows.

Every $p\in\mathcal H$ has a unique representation
\[
 p=|p|e^{i\theta},\qquad -\frac{\pi}{2}<\theta<\frac{\pi}{2}.
\]
We take the principal logarithm restricted to $\mathcal H$:
\[
 \operatorname{Log}p=\log|p|+i\theta,
 \qquad p^\rho=\exp\bigl(\rho\operatorname{Log}p\bigr),
 \quad \rho\in\mathbb R.
\]
This branch of the logarithm is holomorphic in $\mathcal H$, with
$(\operatorname{Log}p)'=1/p$. Therefore $p^\gamma$ and
$p^{\gamma-1}$ are holomorphic there. For real $p>0$ we have
$\theta=0$, so these definitions agree with the real powers used in
the calculation for $p>\lambda^{1/\gamma}$.

We must also check that the denominator does not vanish. Since
$0<\gamma<1$ and $|\theta|<\pi/2$, we have
$|\gamma\theta|<\pi/2$. Thus
\[
 \operatorname{Re}(p^\gamma)
       =|p|^\gamma\cos(\gamma\theta)>0,
 \qquad
 \operatorname{Re}(p^\gamma+\lambda)>\lambda>0.
\]
It follows that $p^\gamma+\lambda\ne0$ throughout $\mathcal H$.
Consequently, $R$ is holomorphic in $\mathcal H$.

\medskip
\noindent\emph{We now extend the equality to the whole half-plane.}
Both $F$ and $R$ are holomorphic in the connected open set
$\mathcal H$, and the series calculation above shows that they coincide
for every real $p>\lambda^{1/\gamma}$. Although this equality was
first obtained on a real interval, it is sufficient for the identity
theorem: fix any $q>\lambda^{1/\gamma}$ and consider
\[
 p_j=q+\frac1j,\qquad j=1,2,\ldots.
\]
The distinct zeros $p_j$ of $F-R$ converge to $q$, and $q$ lies
inside $\mathcal H$. The identity theorem therefore implies that
$F-R$ vanishes identically in $\mathcal H$. Hence
\[
 F(p)=\frac{p^{\gamma-1}}{p^\gamma+\lambda},
 \qquad p\in\mathcal H.
\]
Thus the first equality in \eqref{eq:ML-transforms} is proved
for all $\operatorname{Re}p>0$, in particular for every real $p>0$.
Notice that the identity theorem is applied in the open half-plane,
not at the boundary point $p=0$.

\medskip
\noindent\emph{Finally, we prove the second formula.}
Fix $p\in\mathcal H$ and take $0<\varepsilon<R<\infty$.
Using $w'(t)=-\lambda k_{\gamma,\lambda}(t)$ and integrating by parts on
this finite interval, we obtain
\[
 \begin{aligned}
 \int_\varepsilon^R e^{-pt}k_{\gamma,\lambda}(t)\dd t
 &=-\frac1\lambda\int_\varepsilon^R e^{-pt}w'(t)\dd t\\
 &=-\frac1\lambda\bigl[e^{-pt}w(t)\bigr]_\varepsilon^R
       -\frac p\lambda\int_\varepsilon^R e^{-pt}w(t)\dd t.
 \end{aligned}
\]
The boundary terms satisfy
\[
 e^{-p\varepsilon}w(\varepsilon)\longrightarrow1
       \quad(\varepsilon\downarrow0),\qquad
 |e^{-pR}w(R)|\leq e^{-(\operatorname{Re}p)R}
       \longrightarrow0\quad(R\to\infty).
\]
The integrals converge absolutely, as proved at the beginning of
this proof. Hence we may let $\varepsilon\downarrow0$ and
$R\to\infty$ in the integration-by-parts identity above. Using
the first Laplace transform formula just proved, we obtain
\[
 \begin{aligned}
 \int_0^\infty e^{-pt}k_{\gamma,\lambda}(t)\dd t
 &=\frac1\lambda\bigl(1-pF(p)\bigr)\\
 &=\frac1\lambda
        \left(1-\frac{p^\gamma}{p^\gamma+\lambda}\right)
  =\frac1{p^\gamma+\lambda}.
 \end{aligned}
\]
This proves the second equality and completes the proof of
\eqref{eq:ML-transforms}.
\end{proof}

In the eigenfunction expansions below, $\lambda$ is replaced by
$\lambda_n$, where $\lambda_n\to\infty$. No fixed $p>0$ can satisfy
$p>\lambda_n^{1/\gamma}$ for every $n$. Formula
\eqref{eq:ML-transforms}, in contrast, is valid for every $n$ at
one and the same arbitrary $p>0$. It also permits the later study
of the limit $p\downarrow0$. No value of the first Laplace integral
at $p=0$ is asserted or used.

Let $\{\lambda_n\}_{n\in\N}$ be the eigenvalues of $A$, repeated according to multiplicity, and choose an orthonormal eigenbasis $\{\varphi_n\}_{n\in\N}$. Thus
\[
 0<\lambda_1\leq\lambda_2\leq\cdots\longrightarrow\infty.
\]
Let us define fractional powers of the operator $A$ using the von Neumann theorem. Namely, let $\tau$ be a real number. Then
\[
A^\tau f = \sum_{k=1}^\infty \lambda_k^\tau\,  (f,\varphi_k) \,  \varphi_k,
\]
and the domain of the operator $A^\tau$ has the form
\[
 \D(A^\tau)=\left\{f\in L^2(\Omega):
       \sum_{n=1}^{\infty}\lambda_n^{2\tau}
                    |(f,\varphi_n)|^2<\infty\right\}.
\]
For $\tau>0$, $A^{-\tau}$ is bounded on $L^2(\Omega)$.

The scalar series associated with the observation satisfy the following estimate.

\begin{prop}\label{prop1}
For $h,\psi\in L^2(\Omega)$,
\begin{equation}\label{eq:weighted-summability}
 \sum_{n=1}^{\infty}
   |(h,\varphi_n)(\varphi_n,\psi)|
       \leq\|h\|\|\psi\|.
\end{equation}
Consequently, for every $k\in\Nzero$,
\[
 \sum_{n=1}^{\infty}
 \frac{|(h,\varphi_n)(\varphi_n,\psi)|}{\lambda_n^{k+1}}
       \leq\lambda_1^{-k-1}\|h\|\|\psi\|.
\]
In particular, no restriction on the dimension $N$ is needed for these weighted series.
\end{prop}
\begin{proof}
Apply the Cauchy--Schwarz inequality and Parseval's identity to the two sequences of Fourier coefficients. The second estimate follows from $\lambda_n\geq\lambda_1$.
\end{proof}

\begin{lem}\label{vanishing_coefficients}
Let $\eta_1<\cdots<\eta_N$ be real numbers. If
\[
 \sum_{j=1}^N b_jp^{\eta_j}+o(p^{\eta_N})=0
       \qquad\hbox{as }p\downarrow0,
\]
then $b_j=0$ for every $j$.
\end{lem}
\begin{proof}
Dividing by $p^{\eta_1}$ and taking the limit gives $b_1=0$. The remainder divided by $p^{\eta_1}$ tends to zero, including the case $N=1$. Remove the first term and repeat. After finitely many steps all coefficients vanish.
\end{proof}

The existence and uniqueness of the solution of the forward problem follow from the next theorem. We state it for a general right-hand side $F$.

\begin{thm}[ {\cite[Theorem 4.2]{4}, \,\cite[Theorem~9]{7}}]\label{lem2}
Let $0<\alpha<1$, $T>0$, $a\in H_0^1(\Omega)$, and $F\in L^2(0,T;L^2(\Omega))$. Then the problem
\[
 \partial_t^\alpha(u-a)+Au=F,\qquad
 u-a\in H_\alpha(0,T;L^2(\Omega)),\qquad u|_{\partial\Omega}=0,
\]
has a unique solution such that
\[
 u\in L^2(0,T;H^2(\Omega)\cap H_0^1(\Omega)).
\]
Moreover, there exists a constant $C>0$, independent of $a$ and $F$, such that
\begin{equation}\label{eq:forward-regularity}
 \begin{split}
 &\|u-a\|_{H_\alpha(0,T;L^2(\Omega))}
    +\|u\|_{L^2(0,T;H^2(\Omega)\cap H_0^1(\Omega))}\\
 &\hspace{20mm}\leq C\bigl(\|a\|_{H_0^1(\Omega)}
                   +\|F\|_{L^2(0,T;L^2(\Omega))}\bigr).
 \end{split}
\end{equation}
\end{thm}

In particular, the theorem applies to $F(x,t)=\mu(t)f(x)$, since
\[
 \|\mu f\|_{L^2(0,T;L^2(\Omega))}
       =\|\mu\|_{L^2(0,T)}\|f\|.
\]
For this right-hand side, separation of variables gives
\begin{equation}\label{eq:solution-series}
 \begin{split}
 u_{\alpha,\mu,a}(x,t)
  ={}&\sum_{n=1}^{\infty}
       E_\alpha(-\lambda_nt^\alpha)(a,\varphi_n)\varphi_n(x)\\
   &+\sum_{n=1}^{\infty}(f,\varphi_n)\varphi_n(x)
          \int_0^t k_{\alpha,\lambda_n}(t-s)\mu(s)\dd s.
 \end{split}
\end{equation}
Indeed, taking the scalar product of the equation with $\varphi_n$ yields
\[
 \partial_t^\alpha\bigl(u_n-(a,\varphi_n)\bigr)
     +\lambda_nu_n=\mu(t)(f,\varphi_n),\qquad
 u_n-(a,\varphi_n)\in H_\alpha(0,T),
\]
where $u_n(t)=(u(\cdot,t),\varphi_n)$. The scalar solution formula is classical and is found, for example, in  \cite[Theorem~7]{7}. This formula gives the coefficients in \eqref{eq:solution-series}.

We record the continuity and boundedness needed for the inverse problems. By \eqref{eq:kernel-mass},
\[
 \left|\int_0^t k_{\alpha,\lambda_n}(t-s)\mu(s)\dd s\right|
       \leq\frac{\|\mu\|_{L^\infty(0,T)}}{\lambda_n},\qquad 0\leq t\leq T.
\]
Each scalar coefficient in \eqref{eq:solution-series} is continuous in $t$. For the integral term this follows from continuity of translations in $L^1$, after extending the kernel and $\mu$ by zero outside the interval. For $m>\ell$, Parseval's identity gives
\[
 \begin{split}
 &\sup_{0\leq t\leq T}\left\|
       \sum_{n=\ell+1}^{m}
          E_\alpha(-\lambda_nt^\alpha)(a,\varphi_n)\varphi_n\right\|^2
          \leq\sum_{n=\ell+1}^{m}|(a,\varphi_n)|^2,\\
 &\sup_{0\leq t\leq T}\left\|
       \sum_{n=\ell+1}^{m}(f,\varphi_n)\varphi_n
          \int_0^t k_{\alpha,\lambda_n}(t-s)\mu(s)\dd s\right\|^2
       \leq\|\mu\|_{L^\infty(0,T)}^2
          \sum_{n=\ell+1}^{m}\frac{|(f,\varphi_n)|^2}{\lambda_n^2}.
 \end{split}
\]
The right-hand sides tend to zero as $\ell\to\infty$. Consequently, both series converge in $C([0,T];L^2(\Omega))$, and the solution in Theorem~\ref{lem2} has a representative satisfying
\begin{equation}\label{eq:forward-bound}
 u(\cdot,0)=a,\qquad
 \sup_{0\leq t\leq T}\|u(\cdot,t)\|
       \leq\|a\|+\|\mu\|_{L^\infty(0,T)}\|A^{-1}f\|.
\end{equation}
The equality of this representative with the strong solution follows from equality of all Fourier coefficients for almost every $t$. If $\mu\in L^\infty(0,\infty)$, uniqueness on overlapping intervals gives a bounded solution on the whole positive time axis. The same statements hold for $B$.

\begin{rmk}\label{rem:regularity}
The initial condition is imposed on $u-a$, not on $u$. For $\alpha>1/2$, the inclusion \eqref{eq:initial-space} already implies $u(\cdot,0)=a$ in the sense of the trace. For $\alpha\leq1/2$, this inclusion alone does not imply continuity at $t=0$. In the present problem, continuity follows from \eqref{eq:solution-series} and boundedness of $\mu$. The assumption $a\in H_0^1(\Omega)$ is used throughout to apply Theorem~\ref{lem2} in the stated strong-solution class.
\end{rmk}

We now formulate the main results.

\Needspace{12\baselineskip}
\begin{thm}\label{thm1}
Let $\alpha,\beta\in(0,1)\setminus\Q$, $a,b\in H_0^1(\Omega)$, and $f,\psi\in L^2(\Omega)$. Assume that $\mu,\nu\in L^\infty(0,\infty)$ and that, for some $p_0>0$,
\begin{equation}\label{mu}
 e^{p_0t}\mu,\ e^{p_0t}\nu\in L^1(0,\infty).
\end{equation}
Suppose that
\begin{equation}\label{fnot0}
 \int_\Omega f(x)\psi(x)\dd x\ne0
\end{equation}
and that the common observation
\begin{equation}\label{upsiequal}
 \begin{split}
 y(t)&:=\int_\Omega u_{\alpha,\mu,a}(x,t)\psi(x)\dd x\\
     &=\int_\Omega u_{\beta,\nu,b}(x,t)\psi(x)\dd x,
                  \qquad t>0,
 \end{split}
\end{equation}
is not identically zero. Then
\[
 \alpha=\beta,\qquad
 \mu=\nu\quad\hbox{almost everywhere on }(0,\infty).
\]
Moreover, if $P_n$ is the orthogonal projection onto the eigenspace corresponding to any distinct eigenvalue $\xi_n$ of $A$, then
\[
 (P_n(a-b),\psi)=0.
\]
In particular,
\[
 \int_\Omega a(x)\psi(x)\dd x
       =\int_\Omega b(x)\psi(x)\dd x.
\]
\end{thm}

Theorem~\ref{thm1} answers Inverse Problem~1. We emphasize that it does not require $(A^{-1}a,\psi)\ne0$ or any prescribed initial value. The nontriviality condition concerns only the measured function $y$.

For Inverse Problem~2, let $A$ and $B$ be the elliptic Dirichlet operators introduced in Section~\ref{sec:introduction}.

\Needspace{10\baselineskip}
\begin{thm}\label{thm2}
Let $\alpha,\beta\in(0,1)$, $a,b\in H_0^1(\Omega)$, and $f,g,\psi\in L^2(\Omega)$. Suppose that $\mu,\nu\in L^\infty(0,\infty)$ satisfy \eqref{mu}, and assume that
\begin{equation}\label{Aa}
 (A^{-1}a,\psi)\ne0,\qquad (B^{-1}b,\psi)\ne0.
\end{equation}
If
\[
 \int_\Omega u_{A,a,\mu,f,\alpha}(x,t)\psi(x)\dd x
 =\int_\Omega u_{B,b,\nu,g,\beta}(x,t)\psi(x)\dd x,
 \qquad t>0,
\]
then
\[
 \alpha=\beta,\qquad (A^{-1}a,\psi)=(B^{-1}b,\psi).
\]
\end{thm}

The next result uses observations only near $t=0$. Here $\D(A^2)=\{w\in\D(A):Aw\in\D(A)\}$, and $\D(B^2)$ is defined similarly. Thus $\psi\in\D(A^2)\cap\D(B^2)$ means that $\psi$, $A\psi$, and $B\psi$ all belong to $H^2(\Omega)\cap H_0^1(\Omega)$.

\Needspace{10\baselineskip}
\begin{thm}\label{thm3}
Let $\alpha,\beta\in(0,1)$, $a,b\in H_0^1(\Omega)$, and $f,g\in L^2(\Omega)$. Fix $T>0$. Assume that $\mu,\nu\in L^\infty(0,T)$ admit representatives such that
\[
 \mu(t)=\mu_0+O(t),\qquad
 \nu(t)=\nu_0+O(t)\qquad(t\downarrow0),
\]
where $\mu_0,\nu_0\in\R$. Let $\psi\in\D(A^2)\cap\D(B^2)$ and suppose that
\begin{equation}\label{condition_on_m_and_n}
 \begin{split}
 c_A&:=\mu_0(f,\psi)-(a,A\psi)\ne0,\\
 c_B&:=\nu_0(g,\psi)-(b,B\psi)\ne0.
 \end{split}
\end{equation}
If
\[
 \int_\Omega u_{A,a,\mu,f,\alpha}(x,t)\psi(x)\dd x
 =\int_\Omega u_{B,b,\nu,g,\beta}(x,t)\psi(x)\dd x,
 \qquad 0<t<T,
\]
then
\[
 \alpha=\beta,\qquad (a,\psi)=(b,\psi),\qquad c_A=c_B.
\]
\end{thm}

\begin{rmk}
Condition \eqref{mu} is not assumed in Theorem~\ref{thm3} (in particular, since the inverse problem is considered here over a finite time interval). The values $\mu_0$ and $\nu_0$ may be denoted by $\mu(0)$ and $\nu(0)$ after choosing the indicated representatives. Unlike Theorem~\ref{thm1}, Theorems~\ref{thm2} and~\ref{thm3} do not require the orders to be irrational. They determine the order, but do not assert equality of the different elliptic operators or source terms.
\end{rmk}

\section{Proof of Theorem~\ref{thm1}}\label{sec:simultaneous}

\begin{proof}
We apply the Laplace transform directly to the eigenfunction
representation \eqref{eq:solution-series}. For the first solution, this representation is
\[
 \begin{split}
 u_{\alpha,\mu,a}(x,t)
 ={}&\sum_{n=1}^{\infty}
       E_\alpha(-\lambda_nt^\alpha)(a,\varphi_n)\varphi_n(x)\\
 &+\sum_{n=1}^{\infty}(f,\varphi_n)\varphi_n(x)
       \int_0^t k_{\alpha,\lambda_n}(t-s)\mu(s)\dd s.
 \end{split}
\]
The series are understood in $L^2(\Omega)$.
We first justify integration term by term.

Set $M_\mu=\|\mu\|_{L^\infty(0,\infty)}$.
By nonnegativity of the kernel and
\eqref{eq:kernel-mass},
\[
 \left|\int_0^t k_{\alpha,\lambda_n}(t-s)\mu(s)\dd s\right|
 \leq M_\mu\int_0^t k_{\alpha,\lambda_n}(r)\dd r
 \leq\frac{M_\mu}{\lambda_n},\qquad t\geq0.
\]
Together with $|E_\alpha(-\lambda_nt^\alpha)|\leq1$ and Parseval's
identity, this gives
\[
 \begin{aligned}
 \|u_{\alpha,\mu,a}(\cdot,t)\|
 &\leq\|a\|+
       M_\mu\left(\sum_{n=1}^{\infty}
                    \frac{|(f,\varphi_n)|^2}{\lambda_n^2}\right)^{1/2}\\
 &\leq\|a\|+\frac{M_\mu}{\lambda_1}\|f\|,
 \qquad t\geq0.
 \end{aligned}
\]
Consequently, for any fixed real $p>0$,
\[
 \int_0^\infty e^{-pt}\|u_{\alpha,\mu,a}(\cdot,t)\|\dd t
   \leq\frac1p\left(\|a\|+\frac{M_\mu}{\lambda_1}\|f\|\right)
   <\infty.
\]
Thus $\widehat u_{\alpha,\mu,a}(p)$ exists as an $L^2(\Omega)$-valued
integral.

The same estimates and Parseval's identity show that the partial
sums of the two series in \eqref{eq:solution-series} are bounded
in $L^2(\Omega)$ by $\|a\|$ and
$M_\mu\lambda_1^{-1}\|f\|$, respectively, uniformly in $t\geq0$
and in the number of terms. These partial sums converge in
$L^2(\Omega)$ for each fixed $t$. For every fixed $p>0$,
multiplication by $e^{-pt}$ provides integrable majorants on
$(0,\infty)$. Therefore, the dominated convergence theorem
for Bochner integrals justifies termwise Laplace integration
of both series in \eqref{eq:solution-series}.

Applying Lemma \ref{lem:laplace-ML} and formula (\ref{conLap}) to \eqref{eq:solution-series}, we obtain
\[
 \begin{split}
 \widehat u_{\alpha,\mu,a}(p)
 ={}&\sum_{n=1}^{\infty}
       \frac{p^{\alpha-1}}{p^\alpha+\lambda_n}
                   (a,\varphi_n)\varphi_n\\
 &+\widehat\mu(p)\sum_{n=1}^{\infty}
       \frac{1}{p^\alpha+\lambda_n}
                   (f,\varphi_n)\varphi_n,
 \qquad p>0.
 \end{split}
\]
Similarly, for the second solution we obtain
\[
 \begin{split}
 \widehat u_{\beta,\nu,b}(p)
 ={}&\sum_{n=1}^{\infty}
       \frac{p^{\beta-1}}{p^\beta+\lambda_n}
                   (b,\varphi_n)\varphi_n\\
 &+\widehat\nu(p)\sum_{n=1}^{\infty}
       \frac{1}{p^\beta+\lambda_n}
                   (f,\varphi_n)\varphi_n,
 \qquad p>0.
 \end{split}
\]

Set
\[
 \begin{aligned}
 a_n&=(a,\varphi_n)(\varphi_n,\psi),\qquad
 b_n=(b,\varphi_n)(\varphi_n,\psi),\\
 f_n&=-(f,\varphi_n)(\varphi_n,\psi).
 \end{aligned}
\]
By the Cauchy--Schwarz inequality and Parseval's identity,
\[
 \begin{aligned}
 \sum_{n=1}^{\infty}|a_n|
 &\leq\left(\sum_{n=1}^{\infty}|(a,\varphi_n)|^2\right)^{1/2}
       \left(\sum_{n=1}^{\infty}|(\varphi_n,\psi)|^2\right)^{1/2}\\
 &=\|a\|\,\|\psi\|.
 \end{aligned}
\]
The same estimate applies to $b$ and $f$, and hence
\[
 \sum_{n=1}^{\infty}(|a_n|+|b_n|+|f_n|)
       \leq(\|a\|+\|b\|+\|f\|)\|\psi\|<\infty.
\]
In particular, all the scalar series below converge absolutely.
For example,
\[
 \sum_{n=1}^{\infty}\frac{|a_n|}{p^\alpha+\lambda_n}
       \leq\frac{\|a\|\,\|\psi\|}{p^\alpha+\lambda_1},
 \qquad p>0.
\]

Finally, by the observation identity \eqref{upsiequal},
\[
 (u_{\alpha,\mu,a}(\cdot,t),\psi)
       =(u_{\beta,\nu,b}(\cdot,t),\psi),\qquad t>0.
\]
The two scalar functions are bounded, so their Laplace transforms
exist for $p>0$. Taking the scalar product with $\psi$ commutes with
the $L^2(\Omega)$-valued integral, because
\[
 |(u(\cdot,t),\psi)|\leq\|u(\cdot,t)\|\,\|\psi\|.
\]
Thus the observation identity yields
\[
 (\widehat u_{\alpha,\mu,a}(p),\psi)
       =(\widehat u_{\beta,\nu,b}(p),\psi),\qquad p>0.
\]
Substituting the two eigenfunction expansions and using the
definitions of $a_n$, $b_n$, and $f_n$, we arrive at
\begin{equation}\label{laplace}
 \begin{split}
 p^{\alpha-1}\sum_{n=1}^{\infty}\frac{a_n}{p^\alpha+\lambda_n}
 -\widehat\mu(p)\sum_{n=1}^{\infty}\frac{f_n}{p^\alpha+\lambda_n}
 ={}&p^{\beta-1}\sum_{n=1}^{\infty}\frac{b_n}{p^\beta+\lambda_n}\\
 &-\widehat\nu(p)\sum_{n=1}^{\infty}
                    \frac{f_n}{p^\beta+\lambda_n},
 \qquad p>0.
 \end{split}
\end{equation}

In what follows, $p$ tends to zero through positive values.

Define the inverse spectral moments
\begin{equation}\label{r0}
 \begin{aligned}
 q_k&=(-1)^k\sum_{n=1}^{\infty}\frac{a_n}{\lambda_n^{k+1}},
 &d_k&=(-1)^k\sum_{n=1}^{\infty}\frac{b_n}{\lambda_n^{k+1}},\\
 r_k&=(-1)^k\sum_{n=1}^{\infty}\frac{f_n}{\lambda_n^{k+1}},
 &&k\in\Nzero.
 \end{aligned}
\end{equation}
All these series converge absolutely. For $z\geq0$ and $M\in\Nzero$, the exact identity
\[
 \frac1{\lambda_n+z}
 =\sum_{k=0}^M\frac{(-z)^k}{\lambda_n^{k+1}}
   +\frac{(-z)^{M+1}}{\lambda_n^{M+1}(\lambda_n+z)}
\]
has a remainder bounded by $z^{M+1}\lambda_1^{-M-2}$. It can therefore be summed against any of the three absolutely summable sequences above. Since \eqref{mu} implies boundedness of $\widehat\mu$ and $\widehat\nu$ near zero, \eqref{laplace} gives
\begin{equation}\label{overdefcon3}
 \begin{split}
 &\sum_{k=0}^{M}q_kp^{\alpha(k+1)-1}
      -\widehat\mu(p)\sum_{j=0}^{m}r_jp^{\alpha j}\\
 &\quad=\sum_{k=0}^{M'}d_kp^{\beta(k+1)-1}
      -\widehat\nu(p)\sum_{j=0}^{m'}r_jp^{\beta j}
      +O(p^{\alpha(M+2)-1})+O(p^{\alpha(m+1)})\\
 &\hspace{42mm}+O(p^{\beta(M'+2)-1})+O(p^{\beta(m'+1)}),
       \qquad p\downarrow0.
 \end{split}
\end{equation}
Here the four nonnegative truncation indices can be chosen independently.

Assumption \eqref{mu} also implies that $\widehat\mu$ and $\widehat\nu$ are holomorphic in
\[
 \mathcal H_{p_0}=\{p\in\C:\operatorname{Re}p>-p_0\}.
\]
For completeness, if $K\subset\mathcal H_{p_0}$ is compact, there is $\delta>0$ such that $\operatorname{Re}p+p_0\geq\delta$ on $K$. For each $\ell\geq0$,
\[
 t^\ell |e^{-pt}\mu(t)|
 \leq t^\ell e^{-\delta t}e^{p_0t}|\mu(t)|,
\]
whose right-hand side is integrable, uniformly for $p\in K$. Differentiation under the integral is therefore valid at every order. Hence
\[
 \widehat\mu(p)=\sum_{\ell=0}^{N}c_\ell p^\ell+O(p^{N+1}),
 \qquad
 \widehat\nu(p)=\sum_{\ell=0}^{N'}c'_\ell p^\ell+O(p^{N'+1}),
\]
where
\[
 c_\ell=\frac{(-1)^\ell}{\ell!}\int_0^\infty t^\ell\mu(t)\dd t,
 \qquad
 c'_\ell=\frac{(-1)^\ell}{\ell!}\int_0^\infty t^\ell\nu(t)\dd t.
\]
These coefficients are moments over the entire positive time axis; in particular, $c_0=\widehat\mu(0)$ is not the initial value $\mu(0)$.

Substituting the Taylor expansions into \eqref{overdefcon3} yields
\begin{equation}\label{main}
 \begin{split}
 &\sum_{k=0}^{M}q_kp^{\alpha(k+1)-1}
       -\sum_{k=0}^{M'}d_kp^{\beta(k+1)-1}\\
 &\quad-\sum_{\ell=0}^{N}\sum_{j=0}^{m}
                      c_\ell r_jp^{\alpha j+\ell}
       +\sum_{\ell=0}^{N'}\sum_{j=0}^{m'}
                      c'_\ell r_jp^{\beta j+\ell}
       =O(p^\theta),
 \end{split}
\end{equation}
where
\[
 \theta=\min\{\alpha(M+2)-1,\ \alpha(m+1),\ N+1,
              \ \beta(M'+2)-1,\ \beta(m'+1),\ N'+1\}.
\]
All six truncation indices can be chosen independently.

\medskip
\noindent\textit{First step.}
First, let us explain how the coefficients in \eqref{main} are compared. Powers coming from different orders may coincide, even when both orders are irrational. All terms with the same exponent must therefore be collected before their coefficients are compared.

Fix $R\in\R$. Choose all six truncation indices so large that $\theta>R$. These choices include every term in the full exponent families with exponent at most $R$. There are only finitely many such terms: for example, $\alpha j+\ell\leq R$ with $j,\ell\geq0$ bounds both indices. The remainder in \eqref{main} is $o(p^R)$, and every retained term whose exponent exceeds $R$ is also $o(p^R)$. Thus, after grouping equal powers, we obtain a finite relation
\[
 \sum_{\eta\leq R} C_\eta p^\eta+o(p^R)=0,
\]
where the exponents in the sum are distinct. If the sum is nonempty, its largest exponent is at most $R$, so Lemma~\ref{vanishing_coefficients} applies and gives $C_\eta=0$ for each exponent in the sum. Since $R$ is arbitrary, every grouped coefficient vanishes. This is a finite-order argument at each prescribed exponent; it does not require passage to a limit in the truncation indices for a fixed positive $p$.

For one fixed irrational order $\gamma$, there are, however, no repetitions within its own expansion. Indeed,
\[
 \gamma j+\ell=\gamma j'+\ell'
 \quad\Longrightarrow\quad
 \gamma(j-j')=\ell'-\ell
 \quad\Longrightarrow\quad (j,\ell)=(j',\ell').
\]
The initial-data exponents $\gamma(k+1)-1$ are also pairwise distinct. Finally, an equality
\[
 \gamma(k+1)-1=\gamma j+\ell
\]
would give $\gamma(k+1-j)=\ell+1$. If $k+1-j=0$, this is impossible because $\ell+1>0$. Otherwise it makes $\gamma$ rational, again a contradiction. Therefore a nonzero term from one observation must be matched by a term from the other observation with exactly the same exponent.

\medskip
\noindent\textit{Second step.}
Let
\[
 0<\xi_1<\xi_2<\cdots
\]
be the distinct eigenvalues of $A$, and let $P_n$ be the orthogonal projection onto $\ker(A-\xi_nI)$. For $x\in L^2(\Omega)$, put
\[
 w_n(x)=\ip{P_nx}{\psi},\qquad
 m_k(x)=\sum_{n=1}^{\infty}\frac{w_n(x)}{\xi_n^{k+1}},
       \qquad k\in\Nzero.
\]
Orthogonality gives $\ip{P_nx}{\psi}=\ip{P_nx}{P_n\psi}$. Consequently,
\[
 \sum_{n=1}^{\infty}|w_n(x)|
 \leq\sum_{n=1}^{\infty}\|P_nx\|\|P_n\psi\|
 \leq\|x\|\|\psi\|.
\]
If some $w_n(x)$ is nonzero, let $n_0$ be its smallest nonzero index. Then
\begin{equation}\label{eq:moment-limit}
 \xi_{n_0}^{k+1}m_k(x)
 =w_{n_0}(x)+\sum_{n>n_0}w_n(x)
                   \left(\frac{\xi_{n_0}}{\xi_n}\right)^{k+1}
 \longrightarrow w_{n_0}(x)\ne0.
\end{equation}
The limit follows by dominated convergence: the absolute values are dominated by the summable sequence $|w_n(x)|$, and each ratio is strictly less than one. In particular, for sufficiently large $k$, the left-hand side has absolute value at least $|w_{n_0}(x)|/2$. Thus $m_k(x)\ne0$ for every sufficiently large $k$.

Thus either all $w_n(x)$ vanish, or $m_k(x)\ne0$ for every sufficiently large $k$. In particular,
\begin{equation}\label{eq:moment-injectivity}
 m_k(x)=0\quad\hbox{for all }k\geq0
 \quad\Longrightarrow\quad
 \ip{P_nx}{\psi}=0\quad\hbox{for all }n.
\end{equation}
Grouping the repeated eigenvalues in \eqref{r0} gives
\[
 q_k=(-1)^km_k(a),\qquad
 d_k=(-1)^km_k(b),\qquad
 r_k=(-1)^{k+1}m_k(f).
\]
Since $\ip{f}{\psi}=\sum_nw_n(f)\ne0$, at least one $w_n(f)$ is nonzero. Hence
\begin{equation}\label{r-tail}
 r_j\ne0\qquad\hbox{for every sufficiently large }j.
\end{equation}
The use of distinct eigenvalues here also takes account of possible multiplicities.

\medskip
\noindent\textit{Third step.}
We need the following elementary fact. If $0<\alpha<\beta<1$ and $\beta$ is irrational, then, for a fixed $s\in\R$, there cannot exist integers $K_j\geq0$ and $L_j\geq-1$ such that
\begin{equation}\label{matching}
 \alpha j+s=\beta K_j+L_j
       \qquad\hbox{for every integer }j\geq j_0.
\end{equation}
It is important to emphasize here that $s$ does not depend on $j$, and the equality is assumed to hold for all sufficiently large $j$, with $\alpha$ not necessarily being an irrational number.

Suppose otherwise. Subtracting the identities at consecutive indices gives
\[
 \alpha=\beta(K_{j+1}-K_j)+(L_{j+1}-L_j).
\]
Because $\beta$ is irrational, the map $(K,L)\mapsto\beta K+L$ is injective on $\Z^2$: equality of two such representations would imply $\beta(K-K')=L'-L$, forcing $K=K'$ and $L=L'$. Therefore there are fixed integers $K,L$ such that
\[
 K_{j+1}-K_j=K,\qquad L_{j+1}-L_j=L
       \qquad(j\geq j_0).
\]
It follows that
\[
 K_j=K_{j_0}+(j-j_0)K,\qquad
 L_j=L_{j_0}+(j-j_0)L.
\]
Their lower bounds imply $K\geq0$ and $L\geq0$. Thus $\alpha=\beta K+L$. Since $0<\alpha<1$, we must have $L=0$; since $\alpha>0$, we must have $K\geq1$. This gives $\alpha=\beta K\geq\beta$, a contradiction. The assertion is proved.

\medskip
\noindent\textit{Fourth step.}
Assume first that $\alpha<\beta$. Suppose that $\widehat\mu$ is not identically zero in a neighborhood of $0$. There is then an index $\ell_0\geq0$ with $c_{\ell_0}\ne0$. By \eqref{r-tail}, every sufficiently large $j$ gives the nonzero coefficient
\[
 -c_{\ell_0}r_j
\]
at the exponent $\alpha j+\ell_0$ in the contribution of the $\alpha$-observation. By Step~1, this contribution cannot be cancelled by a different term from the same observation. Its exponent must therefore also occur in the $\beta$-observation.

Every exponent of the latter observation is of the form $\beta K+L$ with integers $K\geq0$ and $L\geq-1$. The initial-data exponents correspond to $K=k+1$, $L=-1$; the source exponents correspond to $K=j$, $L=\ell\geq0$. Consequently, for all sufficiently large $j$,
\[
 \alpha j+\ell_0=\beta K_j+L_j
\]
for suitable integers with those lower bounds. Step~3 excludes this possibility. It follows that $\widehat\mu$ is zero near $0$. Holomorphy in the connected half-plane $\mathcal H_{p_0}$ and the identity theorem imply $\widehat\mu\equiv0$ there. Uniqueness of the Laplace transform gives
\begin{equation}\label{eq:mu-zero}
 \mu=0\qquad\hbox{almost everywhere on }(0,\infty).
\end{equation}
One way to see this last implication directly is to fix $\sigma>0$. The values $\widehat\mu(\sigma+i\tau)$ form the Fourier transform of the $L^1(\R)$ function obtained by extending $e^{-\sigma t}\mu(t)$ by zero to $t<0$. Fourier-transform injectivity gives \eqref{eq:mu-zero}.

Suppose next that $\ip{P_na}{\psi}\ne0$ for at least one $n$. Step~2 then gives $q_k\ne0$ for every sufficiently large $k$. Step~1 requires each exponent $\alpha(k+1)-1$ to occur in the $\beta$-observation. With $j=k+1$, this gives
\[
 \alpha j-1=\beta K_j+L_j
\]
for every sufficiently large $j$, contrary to Step~3 with $s=-1$. Hence
\[
 \ip{P_na}{\psi}=0\qquad\hbox{for every }n.
\]
Together with \eqref{eq:mu-zero}, the eigenfunction representation implies
\[
 \ip{u_{\alpha,\mu,a}(t)}{\psi}=0\qquad(t>0).
\]
This contradicts the nontriviality of the common observation. Thus $\alpha<\beta$ is impossible. Interchanging $(\alpha,\mu,a)$ and $(\beta,\nu,b)$ excludes $\beta<\alpha$, and therefore
\begin{equation}\label{orders}
 \alpha=\beta.
\end{equation}

\medskip
\noindent\textit{Fifth step.}
Set $\gamma=\alpha=\beta$. Within the expansion for this one irrational order, the two kinds of exponents are disjoint, and the source exponent $\gamma j+\ell$ uniquely specifies $(j,\ell)$. The grouped coefficient rule applied to \eqref{main} therefore gives
\begin{equation}\label{coefficients}
 q_k-d_k=0\quad(k\geq0),\qquad
 (c_\ell-c'_\ell)r_j=0\quad(j,\ell\geq0).
\end{equation}
Fix any $j$ sufficiently large that $r_j\ne0$. The second equality gives $c_\ell=c'_\ell$ for every $\ell$, so the Taylor series of $\widehat\mu$ and $\widehat\nu$ at zero coincide. These transforms agree near zero, hence throughout $\mathcal H_{p_0}$ by the identity theorem. Their injectivity yields $\mu=\nu$ almost everywhere.

The first equality in \eqref{coefficients} says $m_k(a-b)=0$ for every $k$. Applying \eqref{eq:moment-injectivity} to $x=a-b$ gives
\[
 \ip{P_n(a-b)}{\psi}=0\qquad\hbox{for every }n.
\]
The scalar series is absolutely convergent, and summing gives $\ip{a-b}{\psi}=0$. The proof of Theorem~\ref{thm1} is complete.
\end{proof}

\begin{rmk}[The zero-observation case]\label{rem:zero-observation}
Under the source and observation assumptions of Theorem~\ref{thm1}, for each fixed irrational $\alpha$,
\begin{equation}\label{eq:zero-characterization}
 \ip{u_{\alpha,\mu,a}(t)}{\psi}\equiv0
 \quad\Longleftrightarrow\quad
 \mu=0\ \hbox{a.e.},\qquad
 \ip{P_na}{\psi}=0\ \hbox{for every }n.
\end{equation}
To prove the forward implication, expand the Laplace transform of the zero observation as in the proof. Step~1 now applies to the expansion of a single observation, whose exponents are distinct. It yields $q_k=0$ and $c_\ell r_j=0$ for every index. Step~2 gives $r_j\ne0$ for some $j$, hence all $c_\ell$ vanish, which implies $\mu=0$ almost everywhere. The same step applied to the zero moments of $a$ gives $\ip{P_na}{\psi}=0$ for all $n$. The reverse implication is immediate from the spectral formula.

Thus $y\not\equiv0$ is equivalent to $\mu\not\equiv0$ in the almost-everywhere sense or to the existence of a nonzero weight $\ip{P_na}{\psi}$. In particular, a nonzero source factor is sufficient; it need not have a nonzero initial value or a nonzero integral. Without $y\not\equiv0$, equality of two observations still implies equality of the source factors and of the observable initial weights. Unequal irrational orders can occur only in the zero-observation case, when both source factors vanish and both sets of initial weights vanish.

For example, take two orthonormal eigenvectors $\varphi_1,\varphi_2$, set $f=\psi=\varphi_1$, $a=b=\varphi_2$, and $\mu=\nu=0$. Both observations vanish for any two orders, although $\ip{f}{\psi}=1$. This shows why a nontriviality requirement is necessary for order identification.
\end{rmk}

\begin{rmk}[The recoverable part of the initial value]\label{rem:observable-initial}
The theorem determines the initial value modulo
\[
 \mathcal N_\psi=
 \{x\in L^2(\Omega):\ip{P_nx}{\psi}=0\ \hbox{for every }n\}
 =\left(\overline{\operatorname{span}\{P_n\psi:n\in\N\}}\right)^\perp.
\]
In general, one scalar weight on each eigenspace does not determine the entire projection onto that eigenspace. However, if all eigenvalues are simple and $(\varphi_n,\psi)\ne0$ for every $n$, then $\mathcal N_\psi=\{0\}$, and the theorem gives $a=b$. Therefore one should distinguish the general spectral-weight conclusion from this additional observability condition.
\end{rmk}

\begin{exa}[Nonuniqueness at the classical order]
\label{ex:classical}
Let $X=L^2(0,1)$ and let
\[
 A=-\frac{d^2}{dx^2},
 \qquad
 \mathcal D(A)=H^2(0,1)\cap H_0^1(0,1).
\]
Throughout this example, $(\cdot,\cdot)$ and $\|\cdot\|$ denote
the scalar product and norm in $L^2(0,1)$.
Consider the initial-boundary value problem
\[
 \begin{cases}
  u_t+Au=\mu(t)f(x),&0<x<1,\quad t>0,\\[1mm]
  u(0,t)=u(1,t)=0,&t>0,\\[1mm]
  \displaystyle\lim_{t\downarrow0}
      \|u(\cdot,t)-a\|=0,
 \end{cases}
\]
where $f\in L^2(0,1)$ is prescribed and the initial value
$a\in H_0^1(0,1)$ is unknown.

For the difference of two solutions corresponding to the same
order, equality of their observations becomes a zero observation,
while the initial value of the difference need not vanish.
For a prescribed weight $\psi\in L^2(0,1)$, we therefore ask whether
\[
 (u(\cdot,t),\psi)=0\quad\text{for all }t>0
\]
implies $\mu=0$, provided that $(f,\psi)\neq0$.
At the classical order $\alpha=1$, the answer is negative.

Indeed, set
\[
 g(x)=x(1-x),\qquad
 \psi(x)=1-5x(1-x),
\]
and choose
\[
 a(x)=g(x),\qquad
 f(x)=-g(x)-g''(x)=2-x(1-x),\qquad
 \mu(t)=e^{-t}.
\]
Then
\[
 u(x,t)=e^{-t}g(x)
\]
is a solution of the above problem. In fact, $g\in\mathcal D(A)$,
$g(0)=g(1)=0$, and
\[
 \begin{aligned}
 u_t+Au&=e^{-t}(-g-g'')=\mu(t)f(x),\\
 \|u(\cdot,t)-a\|&=|e^{-t}-1|\,\|g\|\longrightarrow0
 \end{aligned}
\]
as $t\downarrow0$.
Moreover,
\[
 \begin{aligned}
 (g,\psi)
 &=\int_0^1x(1-x)\bigl(1-5x(1-x)\bigr)\,dx
   =\frac16-\frac5{30}=0,\\
 \int_0^1\psi(x)\,dx&=1-\frac56=\frac16.
 \end{aligned}
\]
Consequently,
\[
 \begin{aligned}
 (u(\cdot,t),\psi)&=e^{-t}(g,\psi)=0,\qquad t>0,\\
 (f,\psi)&=2\int_0^1\psi(x)\,dx-(g,\psi)=\frac13\neq0.
 \end{aligned}
\]
Nevertheless, $\mu(t)=e^{-t}\not\equiv0$.
This source factor belongs to $L^\infty(0,\infty)$ and satisfies
condition~\eqref{mu}, for example with $p_0=1/2$, since
\[
 \int_0^\infty e^{t/2}|\mu(t)|\,dt
 =\int_0^\infty e^{-t/2}\,dt=2<\infty.
\]
Thus a zero observation does not force the temporal source factor
to vanish when the initial value is unknown.

The same example also gives nonuniqueness with a nonzero common
observation. To see this, put
\[
 z(x,t)=e^{-\pi^2t}\sin(\pi x),\qquad
 U=z+2u,\qquad V=z+u.
\]
Since $z_t+Az=0$, the functions $U$ and $V$ satisfy
\[
 U_t+AU=2e^{-t}f,\qquad V_t+AV=e^{-t}f
\]
with the same homogeneous Dirichlet boundary conditions and the
respective initial values
\[
 a_1(x)=\sin(\pi x)+2g(x),\qquad
 a_2(x)=\sin(\pi x)+g(x).
\]
Both initial values are nonzero and belong to $\mathcal D(A)$.
Since $U-V=u$ and $(u(\cdot,t),\psi)=0$, their observations coincide:
\[
 \begin{aligned}
 (U(\cdot,t),\psi)=(V(\cdot,t),\psi)
 &=e^{-\pi^2t}\int_0^1\sin(\pi x)\psi(x)\,dx\\
 &=\frac{2(\pi^2-10)}{\pi^3}e^{-\pi^2t}\neq0,
 \qquad t>0.
 \end{aligned}
\]
However, their temporal source factors $2e^{-t}$ and $e^{-t}$ are
different. Both are positive, bounded, and satisfy
\eqref{mu} with $p_0=1/2$.
Therefore, the simultaneous uniqueness conclusion of
Theorem~\ref{thm1} cannot, in general, be extended to
$\alpha=\beta=1$, even when the common observation and both
source factors are nonzero.
\end{exa}
\section{Inverse problem of determining the fractional order}\label{sec:order}

\subsection{The small-Laplace-parameter argument}

\begin{proof}[Proof of Theorem~\ref{thm2}]
Let $\{\lambda_n,\varphi_n\}$ and $\{\rho_n,\chi_n\}$ denote the
eigenvalues and orthonormal eigenfunctions of $A$ and $B$,
respectively. The eigenvalues are repeated according to multiplicity,
and $\lambda_n\geq\lambda_1>0$, $\rho_n\geq\rho_1>0$.
All scalar products and norms below are in $L^2(\Omega)$.

In terms of these eigenfunctions, condition~\eqref{Aa} reads
\[
 Q_A:=\sum_{n=1}^{\infty}
       \frac{(a,\varphi_n)(\varphi_n,\psi)}{\lambda_n}\neq0,
 \qquad
 Q_B:=\sum_{n=1}^{\infty}
       \frac{(b,\chi_n)(\chi_n,\psi)}{\rho_n}\neq0.
\]
These series converge absolutely. Indeed, by the Cauchy--Schwarz
inequality and Parseval's identity,
\[
 \sum_{n=1}^{\infty}
    |(h,\varphi_n)(\varphi_n,\psi)|
 \leq\|h\|\,\|\psi\|,\qquad h\in L^2(\Omega),
\]
and the same estimate holds for the basis $\{\chi_n\}$.

Denote the common observation by
\[
 y(t)=(u_{A,a,\mu,f,\alpha}(\cdot,t),\psi)
     =(u_{B,b,\nu,g,\beta}(\cdot,t),\psi),\qquad t>0.
\]
As in the derivation of (3.1), the eigenfunction representation
and Lemma~\ref{lem:laplace-ML} give, for every $p>0$,
\[
 \begin{split}
 p^{1-\alpha}\widehat y(p)
 ={}&\sum_{n=1}^{\infty}
       \frac{(a,\varphi_n)(\varphi_n,\psi)}{p^\alpha+\lambda_n}\\
 &+p^{1-\alpha}\widehat\mu(p)
       \sum_{n=1}^{\infty}
       \frac{(f,\varphi_n)(\varphi_n,\psi)}{p^\alpha+\lambda_n}.
 \end{split}
\]
We now let $p\downarrow0$. The first series tends to $Q_A$
by dominated convergence, since
\[
 \frac{|(a,\varphi_n)(\varphi_n,\psi)|}{p^\alpha+\lambda_n}
 \leq\frac{|(a,\varphi_n)(\varphi_n,\psi)|}{\lambda_n},
\]
and the majorants are summable.

Condition~\eqref{mu} implies $\mu,\nu\in L^1(0,\infty)$.
In particular, $|\widehat\mu(p)|\leq\|\mu\|_{L^1(0,\infty)}$
for $p>0$. Hence the absolute value of the second term is at most
\[
 p^{1-\alpha}\|\mu\|_{L^1(0,\infty)}
       \frac{\|f\|\,\|\psi\|}{\lambda_1}
 \longrightarrow0,
\]
because $\alpha<1$. Applying the same argument to the second
solution, with the eigenfunctions $\{\chi_n\}$, we obtain
\begin{equation}\label{eq:leading-laplace}
 \lim_{p\downarrow0}p^{1-\alpha}\widehat y(p)=Q_A,
 \qquad
 \lim_{p\downarrow0}p^{1-\beta}\widehat y(p)=Q_B.
\end{equation}

Suppose that $\alpha<\beta$. Then
\[
 p^{1-\alpha}\widehat y(p)
 =p^{\beta-\alpha}
       \bigl(p^{1-\beta}\widehat y(p)\bigr)
 \longrightarrow0,
\]
since the expression in parentheses tends to the finite number
$Q_B$. This contradicts the first limit in
\eqref{eq:leading-laplace} and $Q_A\neq0$.
Interchanging the two solutions excludes $\beta<\alpha$.
Therefore $\alpha=\beta$.

For this common order, the two expressions on the left in
\eqref{eq:leading-laplace} coincide. Their limits are therefore
equal, so $Q_A=Q_B$. This is precisely the remaining conclusion
of the theorem.
\end{proof}

\begin{rmk}
The proof of Theorem~\ref{thm2} only needs the source transforms to remain bounded as $p\downarrow0$. Under the other standing hypotheses, $\mu,\nu\in L^1(0,\infty)$ already suffice for this argument. The stronger common assumption \eqref{mu} is retained in its statement for consistency with the main theorem.
\end{rmk}

\subsection{The short-time argument}

\begin{proof}[Proof of Theorem~\ref{thm3}]
We first find the short-time expansion of the observation for the equation with $A$. Write
\[
 u(t)=S_\alpha(t)a+V_\alpha(t),
\]
where these are the homogeneous and inhomogeneous parts of \eqref{eq:solution-series}. The exact Mittag--Leffler identity
\[
 E_{\alpha,1}(z)=1+\frac{z}{\Gamma(1+\alpha)}
                         +z^2E_{\alpha,1+2\alpha}(z)
\]
implies
\begin{equation}\label{eq:initial-asymptotic}
 \ip{S_\alpha(t)a}{\psi}
 =\ip{a}{\psi}
    -\frac{t^\alpha}{\Gamma(1+\alpha)}\ip{a}{A\psi}
    +R_a(t).
\end{equation}
Indeed,
\[
 R_a(t)=t^{2\alpha}\sum_{n=1}^{\infty}
 \lambda_n^2 E_{\alpha,1+2\alpha}(-\lambda_nt^\alpha)
              (a,\varphi_n)(\varphi_n,\psi).
\]
By \eqref{eq:ML-bound}, self-adjointness, and the Cauchy--Schwarz inequality,
\begin{equation}\label{eq:initial-remainder}
 \begin{split}
 |R_a(t)|
 &\leq C_\alpha t^{2\alpha}
       \sum_{n=1}^{\infty}|(a,\varphi_n)|
                            |(\varphi_n,A^2\psi)|\\
 &\leq C_\alpha t^{2\alpha}\|a\|\|A^2\psi\|.
 \end{split}
\end{equation}
This bound also justifies the rearrangement of the scalar spectral series.

For the source part, use the exact identity
\[
 E_{\alpha,\alpha}(z)=\frac1{\Gamma(\alpha)}
                            +zE_{\alpha,2\alpha}(z).
\]
It gives
\begin{equation}\label{eq:source-asymptotic}
 \ip{V_\alpha(t)}{\psi}
       =(J^\alpha\mu)(t)\ip{f}{\psi}-R_f(t),
\end{equation}
where
\[
 R_f(t)=\int_0^t(t-s)^{2\alpha-1}\mu(s)
       \sum_{n=1}^{\infty}\lambda_n
       E_{\alpha,2\alpha}(-\lambda_n(t-s)^\alpha)
       (f,\varphi_n)(\varphi_n,\psi)\dd s.
\]
For $0<t\leq T_0<T$, the bound on $\mu$ and Proposition~\ref{prop1}, applied with $A\psi$, imply
\begin{equation}\label{eq:source-remainder}
 \begin{split}
 |R_f(t)|
 &\leq C_\alpha\|\mu\|_{L^\infty(0,T_0)}
       \|f\|\|A\psi\|\int_0^t(t-s)^{2\alpha-1}\dd s\\
 &\leq \frac{C_\alpha}{2\alpha}
       \|\mu\|_{L^\infty(0,T_0)}
       \|f\|\|A\psi\| t^{2\alpha}.
 \end{split}
\end{equation}
In particular, the absolute majorant is integrable, so the interchange of the series and the time integral is justified.

Choose $T_0>0$ sufficiently small that $|\mu(s)-\mu_0|\leq C_0s$ for $0<s<T_0$. Then
\[
 \left|(J^\alpha\mu)(t)-\frac{\mu_0t^\alpha}{\Gamma(1+\alpha)}\right|
 \leq\frac{C_0}{\Gamma(\alpha)}
          \int_0^t(t-s)^{\alpha-1}s\dd s
 =\frac{C_0t^{\alpha+1}}{\Gamma(\alpha+2)}.
\]
Combining this estimate with \eqref{eq:initial-asymptotic}--\eqref{eq:source-remainder}, and using $\alpha+1>2\alpha$, yields
\begin{equation}\label{eq:yA-short}
 y_A(t)=\ip{a}{\psi}
       +\frac{c_A t^\alpha}{\Gamma(1+\alpha)}+O(t^{2\alpha}),
       \qquad t\downarrow0.
\end{equation}
Exactly the same argument for $B$ gives
\begin{equation}\label{eq:yB-short}
 y_B(t)=\ip{b}{\psi}
       +\frac{c_B t^\beta}{\Gamma(1+\beta)}+O(t^{2\beta}).
\end{equation}

Taking $t\downarrow0$ in $y_A(t)=y_B(t)$ first proves $\ip{a}{\psi}=\ip{b}{\psi}$. After subtracting these constants, suppose that $\alpha<\beta$ and divide by $t^\alpha$. Equations \eqref{eq:yA-short} and \eqref{eq:yB-short} give
\[
 \frac{c_A}{\Gamma(1+\alpha)}+O(t^\alpha)
 =\frac{c_B t^{\beta-\alpha}}{\Gamma(1+\beta)}
       +O(t^{2\beta-\alpha}).
\]
The right-hand side tends to zero and the left-hand side tends to a nonzero number, a contradiction. The reverse inequality is excluded in the same way. Hence $\alpha=\beta$. Dividing by the common power and taking the limit gives $c_A=c_B$. The proof of Theorem~\ref{thm3} is complete.
\end{proof}

\begin{rmk}[A sufficient condition for \eqref{Aa}]\label{rem:elliptic-positivity}
For the elliptic Dirichlet operators considered in this paper, condition \eqref{Aa} is satisfied under a simple sign assumption. Namely, suppose that
\[
 a\geq0,\quad a\not\equiv0,\qquad
 b\geq0,\quad b\not\equiv0,\qquad
 \psi\geq0,\quad\psi\not\equiv0
 \quad\hbox{in }\Omega.
\]
The functions $w_A=A^{-1}a$ and $w_B=B^{-1}b$ solve the elliptic Dirichlet problems
\[
 Aw_A=a,\qquad Bw_B=b\quad\hbox{in }\Omega,
 \qquad w_A=w_B=0\quad\hbox{on }\partial\Omega.
\]
The positivity-improving property of the Dirichlet resolvent on a connected domain, which follows from the maximum principle in its weak-solution form, gives
\[
 w_A>0,\qquad w_B>0\quad\hbox{almost everywhere in }\Omega.
\]
Consequently,
\[
 (A^{-1}a,\psi)=\int_\Omega w_A(x)\psi(x)\dd x>0,
 \qquad
 (B^{-1}b,\psi)=\int_\Omega w_B(x)\psi(x)\dd x>0.
\]
Thus Theorem~\ref{thm2} applies without a separate verification of \eqref{Aa}. The same nonvanishing conclusion holds if each of $a$, $b$, and $\psi$ is nonzero and has a fixed sign, since their signs can be changed separately. For related maximum-principle arguments in fractional-order identification, see \cite{JK,Yamamoto-arxiv}.
\end{rmk}

\section{Conclusion}\label{sec:conclusion}

We have considered two inverse problems for a time-fractional diffusion equation with an elliptic Dirichlet operator. The main theorem shows that a single nontrivial weighted observation for all positive times determines an irrational fractional order and an exponentially integrable bounded temporal source factor, even when the initial value is unknown. The same observation determines the weighted projections of the initial value onto all eigenspaces. No sign condition or nonvanishing first inverse moment of the initial value is required for this theorem.

For the determination of the fractional order alone, we have obtained two results allowing different elliptic operators, initial values, and source terms. The first uses the leading behavior of the Laplace transform near zero. The second uses the leading fractional term of the observation near $t=0$ and requires data only on a finite time interval. Both results apply to all orders in $(0,1)$ under their respective nonvanishing conditions.

The counterexample at $\alpha=1$ shows that the simultaneous uniqueness statement cannot be extended to the classical equation under the same assumptions. The present paper concerns uniqueness; stability and reconstruction are not addressed. The simultaneous determination problem for rational orders also requires further study.

\end{document}